\documentclass{lms}

\usepackage{enumerate} 
\usepackage{array} 
\usepackage{amsmath,amssymb,amsfonts}
\usepackage{latexsym}
\usepackage{tikz}
\usepackage{graphicx}
\usetikzlibrary{matrix}

\usepackage{tikz-cd}
\usetikzlibrary{positioning}

\everymath{\displaystyle}
\usepackage[all]{xy}
\input xypic
\xyoption{v2}

\newtheorem{theorem}{Theorem}[section]

\newtheorem{cor}[theorem]{Corollary}
\newtheorem{pro}[theorem]{Proposition}
\newtheorem{Ex}[theorem]{Example}
\newtheorem{con}[theorem]{Conjecture}
\newtheorem{question}[theorem]{Question} 

\newtheorem{remark}[theorem]{Remark}

\numberwithin{equation}{section}

\newcommand{\Q}{{\bf Q}}
\newcommand{\Z}{{\bf Z}}

\newcommand{\gal}{\rm Gal}
\newcommand{\tor}{\rm Tor}

\title[Multi-quadratic $p$-rational Number Fields]% end with percent
 {Multi-quadratic $p$-rational Number Fields} % This is the full title of the paper
\author{Y. Benmerieme  and A. Movahhedi}

\classno{11R23 (primary), 11R32, 11R42 (secondary)}
\begin{document}
\maketitle

\begin{abstract}
For each odd prime  $p$, we prove the existence of infinitely many real quadratic fields which are  $p$-rational. Explicit imaginary and real bi-quadratic $p$-rational fields are also given for each prime  $p$. Using a recent method developed by Greenberg, we deduce the existence of Galois extensions of  $\Q$  with Galois group isomorphic to an open subgroup of   $GL_n(\Z_p)$,  for  $n =4$  and  $n =5$  and at least for all the primes $p <192.699.943$.   

\end{abstract}

\thispagestyle{empty}

\section{Introduction} 
Let  $F$  be a number field and $p$ be a prime number. 
Denote by  $S$  the set of primes lying above  $p$  and by  $F_S$  the maximal $p$-extension of  $F$  which is unramified outside the primes above  $S$. The structure of the Galois group  $G(F_S/F)$  reflects the arithmetic properties of the base field  $F$  and it has long  been studied. The number field  $F$  is called  $p$-rational if the Galois group $G(F_S/F)$  is a free pro-$p$-group. The terminology comes from the fact that $p$-rational fields behave in the same way as the field of rational numbers  $\Q$,  with respect to the prime  $p$.
Recently, $p$-rational fields have been revisited by Greenberg \cite{Greenberg16} to study some representations of the absolute Galois group  $G_{\Q}$ of  $\Q$  with large image by a new fundamental (non-geometric, but essentially  $p$-adic) method. 
This work has been continued by Cornut and Ray \cite{Cornut_Ray}. In doing so, Greenberg proposed a new conjecture according to which for any odd prime  $p$, there exist arbitrary large multi-quadratic $p$-rational number fields. If this conjecture holds true for a prime $p$ with complex multi-quadratic fields, then there exists a continuous Galois representation 
$G_{\Q} \to  GL_n(\Z_p)$  
on the general linear group of the ring of  $p$-adic integers  $\Z_p$  with an open image for all degrees  $n \geq 4$. 
%Let us mention here also a recent preprint of Katz \cite{Ka} where he provides motivic Galois representations  
%$\rho : G_{\Q} \to  GL_n(\Q_p)$  
%on the general linear group of the field of  $p$-adic numbers with an open image for any even  $n \geq 6$  and any  $pÊ\equiv 1 \mod 3$  or  $\mod 4$.  
%\cite[Introduction]{Ka}
According to Katz, these representations are "spectacularly non-motivic" \cite[Introduction]{Ka}. For some (motivic) Galois representations  
$\rho : G_{\Q} \to  GL_n(\Q_p)$  
%on the general linear group of the field of  $p$-adic numbers with an open image 
for $n \geq 6$ even and primes  $p \equiv 1 \pmod 3$  or$\pmod 4$  see \cite{Ka}. 

In Section 2, we briefly recall the notion of  $p$-rationality and provide various equivalences for an arbitrary number field to be  $p$-rational with special attention to the case when the number field is totally real. 
In particular, alternative proofs for the necessary and sufficient conditions for a quadratic field to be  $p$-rational are given. 
When $p=3$, a quadratic field $F=\Q(\sqrt{d})$ with $d \equiv -3 \pmod 9$ is  not $3$-rational except when  $d=-3$. 
When $d \not \equiv -3 \pmod 9$, the field  $F$  is  $3$-rational precisely when the $3$-part of the class group of its mirror field  $F'=\Q(\sqrt{-3d})$  is trivial (Proposition \ref{mirror}). In addition,  if the field $F$ is imaginary and its class number is not divisible by $3$, then it is $3$-rational. Using these results, we show the existence of infinitely many imaginary bi-quadratic  $3$-rational fields (Corollary \ref{infinitely bi-quadratic 3-rational}). 
%Then we prove that the degree over $\Q$  of a multi-quadratic $3$-rational number field which contains $\sqrt{-3}$ is at most $8$  (Corollary \ref{p=3 multiquadratic}).
We end this section by proving that for any prime  $p\geq 7$ which does not divide the order of the tame kernel  $K_2(o_F)$ of a real multi-quadratic  $p$-rational field  $F$,  there exist infinitely many totally imaginary quadratic extensions of  $F$  which are  $p$-rational.

For a prime  $p>3$, an imaginary quadratic field  $F$  is  $p$-rational as soon as  $p$  does not divide the class number of  $F$. 
On the other hand a result of Hartung dating back to the seventies \cite{Hartung74} ensures the existence of infinitely many such number fields. 
There therefore exist infinitely many imaginary quadratic fields which are  $p$-rational (even for $p=2$ and $p=3$).  
The real case is rather more involved. 
According to Greenberg's conjecture, there exist infinitely many real quadratic $p$-rational fields. 
In Section 3, we prove the existence of such an infinite family (Theorem \ref{infinitely many real quadratic}). The proof appeals, in particular, to a formula due to Coates expressing  the order of the  $\Z_p$-torsion of the abelianized factor group $G(F_S/F)^{ab}$  in terms of the invariants of the totally real number field  $F$ on the ground level such as class number and discriminant (see formula (\ref{coates}) below), as well as a result of Ono on the existence of some infinitely many fundamental discriminants (see Theorem \ref{Ono}). 
%We conclude this section by showing that if $F$ is a real multi-quadratic $5$-rational field which has tame kernel $K_2(o_F)$ of order prime to $5$, then $F(\sqrt{5})$  is  $5$-rational.

In Section 4, for each odd prime  $p$, we provide an imaginary and a real bi-quadratic number field which are  $p$-rational. Namely, the imaginary bi-quadratic field  $\Q(\sqrt{p(p+2)}, \sqrt{-p})$  and the real bi-quadratic field  $\Q(\sqrt{p(p+2)}, \sqrt{p(p-2)})$ are  $p$-rational as soon as  $p>3$. 
%As a consequence, the field $\Q(\sqrt{p})$ is $p$-rational when $p+2$ or $p-2$ is a square. The same holds when $p-1$ or $p-4$ is square. 
%It is believed that there exist infinitely many such primes (Bunyakovsky's Conjecture). 
%We may then ask if $\Q(\sqrt{p})$ is always $p$-rational. We give a characterisation of its $p$-rationality by ordinary (respectively generalised) Bernoulli numbers when $p  \equiv 1 \pmod 4$ (respectively $p  \equiv 3 \pmod 4$) (Proposition \ref{p-rationality of  Q(p)}). Then we note that our question is equivalent to two old conjectures (Corollary \ref{AAC}), which still remain unsettled.  Namely, the Ankeny-Artin-Chowla conjecture, to be found in \cite[page 480]{AAC52}, as well as the conjecture in a paper of Mordell \cite[page 283]{Mordell61}. So we now have a new interpretation of these conjectures by the $p$-rationality of $\Q(\sqrt{p})$.  
%For $p  \equiv 1 \pmod 4$,
% %the existence of a real quadratic $(p,i)$-regular field, with $i=(p-1)/2$, implies $\Q(\sqrt{p})$ is $p$-rational (see the paragraph following Corollary \ref{AAC}). 
%%Furthermore, 
%we show that if $F$ is a real multi-quadratic field which is both $p$-rational and $(p,i)$-regular, with $i=(p-1)/2$, then $F(\sqrt{p})$ is $p$-rational (Proposition \ref{p=1 mod 4 multiquadratic}). 
%An alternative proof of the last result is given in Corollary \ref{p=5 multiquadratic} (section $3$) for  $p=5$.
Finally, adding the square root of  $-p$  to the above real bi-quadratic field yields 
%a  $p$-rational number field which by the method developed by Greenberg \cite[Proposition 6.1 and Remark 6.8]{Greenberg16}  guarantee the existence of 
continuous Galois representations  
$\rho : G_{\Q} \to  GL_n(\Z_p)$  
with an open image for both  $n =4$  and  $n =5$  and for all the primes $p <192.699.943$.

%\newpage 
We list here the notations used in the text for any number field  $F$. 
They are mostly the usual ones.

 \begin{tabular}[ht]{rl}
%$F$						& a number field; \\
$o_F$					& ring of integers of  $F$; \\
$o_F^\prime=o_F[1/p]$		& ring of  $p$-integers of  $F$; \\
$d_F$					& discriminant of  $F/\Q$; \\
$h_F$					& the class number of  $F$; \\
$r_1$, $r_2$				& number of real (resp. non-conjugate complex)  embeddings of $F$; \\
%$F_{\infty} =\cup F_n$		& cyclotomic $\Z_p$-extension of $F$, with finite layers $F_n$; \\
%$\mu_p$					& group of $p$-th roots of unity; \\
$\mu_{p^n}$, $\mu_{p^\infty}$  & group of $p^n$-th (resp. all $p$-primary) roots of unity; \\
$\mu_p(F)= \mu_p \cap F$  	& group of all  $p$-th roots of unity contained in $F$; \\
$U_F$ 					& group of units in $F$ ; \\
${\bar U}_F =\varprojlim (U_F/p^m)$ 	&  pro-$p$-completion of  $U_F$; \\
$F_v$    					& completion of  $F$ at a prime  $v$   in  $F$; \\
$U_v$    					& group of units in the local field $F_v$ ; \\
$U_v^{(i)}$    				& $i$-th higher unit group in the local field  $F_v$;\\
${\bar U}_v =\varprojlim (U_v/p^m)$       	 &  pro-$p$-completion of  $U_v$ ; \\
$S$                  				& set of all places of $F$ lying above $p$; \\ 
$A_F$                			& $p$-primary part of the class group of $F$; \\
$A'_F$                			& $p$-primary part of the (p)-class group of $F$; \\
${ F_S}$   				& maximal pro-$p$-extension of $F$ unramified outside $S$;\\
${\hat F}$   				& maximal abelian pro-$p$-extension of $F$ unramified outside the \\ &$p$-adic primes;\\
${\tilde F}$   				& compositum of all the $\Z_p$-extensions of $F$;\\
$G_S(F):=\gal(F_S/F)$ 		& Galois group over  $F$  of the maximal $S$-ramified pro-$p$-extensi-\\ &on of $F$; \\
${\cal G}_S(F)$ 			& Galois group over  $F$  of the maximal $S$-ramified extension of $F$; \\
$X_F:=\gal(\hat F/F)$ 		& Galois group over  $F$  of the maximal $S$-ramified abelian pro-$p$-\\ &extension of  $F$; \\
$T_F:=\tor_{\Z_p}(X_F)$ 		& finite torsion subgroup of  $X_F$. 
%$D_F$  					& Leopoldt kernel.
%$X^\circ$  &  maximal finite submodule of  $X_\infty$. \\
\end{tabular}
\vspace{1mm}

We will also need the following notations and conventions:  for any group $M$, the notation  $M/p$  denotes the cokernel of the raising to the  $p$-th power map  $x \mapsto x^p$  on  $M$. For a quadratic field  $F$ with discriminant  $d_F$, we denote by $\chi_{{}_F}$  or  $\chi_{{}_{d_F}}$  the associated quadratic character. 
By a multi-quadratic field we mean a composite of quadratic fields. 
Namely, a Galois extension  $F$  of the field of rational numbers  $\Q$  such that  $\gal(F/\Q)$  is isomorphic to  $(\Z/2)^t$  for some integer  $t\geq1$. The cardinality of a finite set  $A$  is denoted by  $\vert A \vert$.   
%or by  $\# A$. 
\section{$p$-rationality} 
The notion of  $p$-rationality was introduced and studied in \cite {Movahhedi-Nguyen90,Movahhedi88,Movahhedi90}, notably in order to produce infinitely many non abelian number fields satisfying Leopoldt's Conjecture at the prime $p$. 
%We refer the reader to the above papers for an account on the subject as well as  \cite{Greenberg16}. 
A number field  $F$ is called  $p$-rational when the Galois group  $G_S(F)$  is a free  pro-$p$-group. 
In other words,  the second cohomology group  $H^2(G_S(F), \Z/p)$  vanishes.  Such a field satisfies the Leopoldt conjecture at the prime  $p$  since one of the equivalent statements for this conjecture is the vanishing of  the second cohomology group  $H^2(G_S(F), \Q_p/\Z_p)$. 
The terminology takes into account the fact that the field of rationals  $\Q$  is  $p$-rational for all  $p$  
and in some sense a  $p$-rational field behaves like  $\Q$  with respect to  $p$.

Let  $R_2F$  be the  $p$-primary part of the "regular kernel"  (see  \cite [Sections I \& II]{Gras86} and also \cite[Introduction]{GJ89}). If  $p$  is odd or if $F$  has no real embedding then $R_2F$  is the same as the  $p$-primary part of the so called tame kernel  $K_2(o_F)$,  whereas when  $p=2$ and  $F$  possesses a real embedding then the difference between the  $2$-part of $K_2(o_F)$  and  $R_2F$  lies  in the fact that the Hilbert symbols at real places are considered as being tame. A number field  $F$  with trivial  $R_2F$  is called  $p$-regular by Gras and Jaulent \cite{GJ89}. The terminology comes from the fact that the cyclotomic field  $\Q(\mu_p)$  is  $p$-regular exactly when  $p$  is a regular prime. 
For an odd prime  $p$, using Tate's result on  $K_2$  and Galois cohomology, Soul\'e showed  \cite [Lemme 10]{Soule79} that the Chern character realises an isomorphism  
$$K_2(o_F) /p  \cong H^2(Spec \, o_F^\prime, \Z/p(2)) \cong H^2({\cal G}_S(F), \Z/p(2)),$$
where   $\Z/p(2)$  denotes the  $2$-fold Tate twist of   $\Z/p$. Hence the  $p$-regularity of a number field  $F$  can be interpreted by the vanishing of  the second Galois cohomology group  $H^2({\cal G}_S(F), \Z/p(2))$.  
% Here,   $\Z/p(2)$  denotes the  $2$-fold Tate twist of   $\Z/p$. 
Therefore, the two notions of  $p$-regularity and $p$-rationality are of different nature since they correspond to different twists "\`a la Tate". Nevertheless, when the maximal real subfield  $\Q(\mu_p)^+$  of  $\Q(\mu_p)$  is contained in   $F$, then  $F$  is  $p$-rational precisely when  $F$  is  $p$-regular. 
In particular,  $F$  is  $3$-rational precisely when  $3$  does not divide the order of the tame kernel  $K_2(o_F)$.  
%This last assertion will be used later on. 

By the Euler-Poincar\'e characteristic  
$$\chi(G_S(F)) =1-d(G_S(F))+r(G_S(F))=-r_2,$$ 
where  $-r_2$  is the number of complex places of  $F$. Hence the pro-$p$-group  $G_S(F)$  is of rank  $1+r_2$ when it is free and the  $p$-rationality of  $F$  is equivalent to the fact that the abelianized factor group  $X_F:=G_S(F)^{ab} = \gal({\hat F}/F)$  is a free  $\Z_p$-module of rank   $1+r_2$: 
$$X_F \simeq \Z_p^{1+r_2}.$$
In other words a number field  $F$  is  $p$-rational precisely when Leopoldt's conjecture holds for  $F$  at  $p$  and the $\Z_p$-torsion module  $T_F :=\tor_{\Z_p}(X_F)$ is trivial.

%In the 
%%beginning of 
%1980's, 
%%in a series of papers, 
%this finite module  $T_F$ was broadly studied by G. Gras.  
%%\cite{Gras82, Gras83}  
%For instance, 
%%under Leopoldt's conjecture for  $F$  at  $p$, he proves a relationship between this finite module  $T_F$  and the class group of  $F(\mu_p)$  
%%\cite[Theorem I2 and Theorem I3]{Gras82}  
%%which is interesting on its own right. 
An interesting genus formula for this finite module $T_F$  "\`a la Chevalley" is obtained in \cite [Proposition 6]{Gras86}  under Leopoldt's conjecture. For a different proof see  \cite [appendice]{Movahhedi-Nguyen90}. As a consequence, Gras obtains all abelian $p$-extensions  $F$ of  $\Q$ such that  $T_F= 1$    for  $p=2$ and for $p=3$  
\cite [Corollary to Theorem 2]{Gras86}. Since an abelian extension of  $\Q$  satisfies Leopoldt's conjecture, the above fields  $F$  turn out to be those abelian  $p$-extensions of  $\Q$ which are $p$-rational (for $p=2$ or for  $p=3$). 
For an arbitrary odd prime $p$,  Galois  $p$-extensions of  $\Q$ (or an arbitrary totally real number field) which are  $p$-rational are characterised in  \cite [Chapitre III, Proposition 2, p.41]{Movahhedi88} by means of ramification : 
an arbitrary Galois $p$-extension  $F$  of  $\Q$  is  $p$-rational precisely when at most one  non-$p$-adic prime  $\ell$  ramifies in  $F$  and  $\ell$  is inert in the cyclotomic  $\Z_p$-extension of  $\Q$, in other words  
$\ell \not \equiv 1 \pmod {p^2}$. Note that for  $\ell$  to be ramified in a  $p$-extension  $F/\Q$, it is necessary that $\ell \equiv 1 \pmod p$.

%To our knowledge, the relationships between this finite module  $T_F$  and the class group of  $F':=F(\mu_p)$ (under Leopoldt's conjecture for  $F$  at  $p$) dates back to Gras \cite[Theorem I2 and Theorem I3]{Gras82}  which is interesting on its own right: 
%$$ \rk_p (T_F) = s-\delta_F + \rk_p (A'_{F'}^\omega)$$

When the number field  $F$ is totally real, Leopoldt's conjecture for  $F$  at  $p$  is expressed by the non vanishing of the  $p$-adic regulator  $R_{p,F}$  and 
the following remarkable formula 
%(which can be viewed as a special case of the main conjecture of Iwasawa theory) 
linking the order of  $T_F$  with   $R_{p,F}$  was proved by Coates  \cite [Appendix, Lemma 8]{Coates77} 
(see also \cite [Proposition 2.1]{Ng86}) :  assume  $R_{p,F} \neq 0$, then 
\begin{equation}\label{coates}
 \vert T_F \vert  \sim  \vert \mu_{p^\infty} \cap F(\mu_p) \vert  \frac{h_F R_{p,F}}{\sqrt{d_F}} 
 \prod_{v\vert p} (1-(N v)^{-1}),  
\end{equation} 
where  $p$  is any odd prime,  $\sim$  stands for equality up to a  $p$-adic unit and  $N: =N_{F/\Q}$  is the norm in  $F/\Q$. The above formula theoretically  provides a way to test  $p$-rationality of totally real number fields. For instance, suppose the totally real field  $F$  to be quadratic. Then, the above formula becomes 
\begin{equation}\label{sgn}
\vert T_F \vert  \sim   h_F R_{p,F} /p^{1/e},  
\end{equation} 
where  $e:=2$ or  $1$  according as  $p$  ramifies in  $F$  or not. 
Denote by  $\varepsilon$  the fundamental unit of  $F$  so that  $R_{p,F} = \log_p(\varepsilon)$,  where  $\log_p$  is the  $p$-adic logarithm. Then  the  $p$-adic valuation of  
$R_{p,F}$ is the same as that of  ($\varepsilon^{q-1}-1$),  where  $q:=Nv$  and  $v$  is a  $p$-adic prime of  $F$.  Hence,  the real quadratic field  $F$  is  $p$-rational precisely when  $p$  does not divide the class number of  $F$  and 
$\varepsilon^{q-1} \in  U_v^{(1)} \setminus U_v^{(1+e)}$.  
%(note that, the  ${\mathfrak p}_v$-valuation of  $ \vert T_F \vert$  is a multiple of  $e$  hence when  $e=2$  and  $p$  divides  $ \vert T_F \vert$,  then either  $p$  divides  $h_F$  or  $\varepsilon^{q-1} \in U_v^{(1+e)}=U_v^{(3)}$  is a  $p$-th power).    
Now, when  $p>3$  or  $p=3$  and is unramified in  $F$, then  $e <p-1$ and raising to the  $p$-th power realises an isomorphism between $U_v^{(1)}$  and  $U_v^{(1+e)}$  \cite [Prop 9, page 219]{Serre62}.  Hence $\varepsilon^{q-1} \in U_v^{(1+e)}$  precisely when  $\varepsilon^{q-1}$ or   $\varepsilon$  is a  $p$-th power in the completion  $F_v$. 
Therefore, when  $p>3$  or  $p=3$  and is unramified in  $F$, then the real quadratic field  $F$  is  $p$-rational precisely when  $p$  does not divide the class number of  $F$  and the fundamental unit of  $F$  is not a  $p$-th power in  $F_v$. 
A different proof of this last equivalence was given in \cite[Proposition 4.1]{Greenberg16} and we shall provide  another alternative proof below (see Corollary \ref{cor2}).  

%On the other hand, if the totally real number field  $F$  is abelian of degree  $r$  corresponding to a group of Dirichlet characters  $X$, then the $p$-adic class number formula for  $F$ reads 
%%\cite[Theorem 5.24]{Washington97}
%$$ \frac{2^{r-1}h_F R_{p,F}}{\sqrt{d_F}} = \prod_{1\neq \chi \in X} (1-\frac{\chi(p)}{p})^{-1} L_p(1,\chi).$$
%
% as  pressage / foretell / predict of the main conjecture of Iwasawa theory, Coates proved the following interesting formula 

Reconsider the case of a general totally real number field  $F$  of degree  $r$.  Let  $\zeta_{F,p}(s)$  be the  $p$-adic  $\zeta$-function of  $F$.   
The residue of  $\zeta_{F,p}$  at   $s=1$  is given by \cite{Colmez88} 
$$\lim_{s \to 1} (s-1) \zeta_{F,p}(s) = \frac{2^{r-1}h_F R_{p,F}}{\sqrt{d_F}} \prod_{v\vert p} (1-(N v)^{-1})$$ 
and the Leopoldt conjecture for  $F$  at  $p$  (which asserts the non vanishing of  $R_{p,F}$)  is valid precisely when  
$\zeta_{F,p}$  has a simple pole at  $s=1$. 

Under Leopoldt's conjecture for  $F$  at the prime  $p$,  the above formula (\ref{coates}) becomes  
\begin{equation}\label{coates-colmez}
 \vert T_F \vert  \sim  \vert \mu_{p^\infty} \cap F(\mu_p) \vert \;  \lim_{s \to 1} (s-1) \zeta_{F,p}(s).  
 \end{equation} 
Hence, we have the following interpretation of  $p$-rationality in terms of the residue of the  $p$-adic  $\zeta$-function: 
\begin{pro} 
\label{residue}  
Let  $p$  be an odd prime number and  $F$  be any totally real number field. 
Let  $\mu_{p^m} := \mu_{p^\infty} \cap F(\mu_p) $. 
Then  $F$  is  $p$-rational precisely when the  $p$-adic valuation of the residue of the  $p$-adic $\zeta$-function of  $F$  at  $1$  is  $-m$: 
$$\lim_{s \to 1} (s-1) \zeta_{F,p}(s) \sim p^{-m}. $$ 

\end{pro} 
When the totally real number field  $F$  is abelian corresponding to a group of Dirichlet characters  $X$, then 
$$\lim_{s \to 1} (s-1) \zeta_{F,p}(s) = \lim_{s \to 1} (s-1) L_p(s,\chi_{{}_0}) \prod_{\chi_{{}_0} \neq \chi \in X} L_p(1,\chi), $$ 
where  $\chi_{{}_0}$  is the trivial character and  $L_p(\quad ,\chi)$  is the  $p$-adic  $L$-function associated to  $\chi$.  Since the residue of   $L_p(s,\chi_{{}_0})$  at  $s=1$  is  $(1-\frac{1}{p})$, according to the above proposition, $F$  is  $p$-rational precisely when  
\begin{equation}\label{product of p-adic L}
\prod_{\chi_{{}_0} \neq \chi \in X} L_p(1,\chi) \sim p^{-m+1}.
\end{equation} 
Therefore, we have the following special case which will be used further on to prove the existence of infinitely many real quadratic $p$-rational fields. 
\begin{cor} 
\label{Lp is unit}  
Let  $p$  be an odd prime. A  real quadratic number field  $F$  is  $p$-rational precisely when  $L_p(1,\chi_{{}_F})$  is a $p$-adic unit. 
\end{cor} 
This Corollary (\ref{Lp is unit}) together with the above relation (\ref{product of p-adic L}) immediately yields the fact that a real multi-quadratic field is  $p$-rational precisely when it is the case of all its quadratic subfields. For a more general result see  \cite[Proposition 3.6]{Greenberg16}.
\vspace{1mm} 
On the completely algebraic side, we have the following proposition noticed by Nguyen Quang Do \cite [Appendix]{Ng}. 
%Here we
%%slightly improve the statement and 
%offer a proof which slightly simplifies his original proof. 

\begin{pro} 
\label{characterisation}  
A  number field  $F$  is  $p$-rational precisely when the following three conditions are satisfied: \\
(i) The $p$-Hilbert  class field  $H \subseteq {\tilde F}$; \\
(ii) The natural map  $\mu_p(F) \to  \oplus_{v\vert p}\mu_p(F_v)$  is an isomorphism;  \\
(iii) The natural map  $U_F/p  \to  \oplus_{v\vert p}U_v/p$  is injective. 
\end{pro}

\begin{proof}
By definition,  the Leopoldt kernel  $D_F$  is the kernel of the diagonal map  ${U_F} \rightarrow  \oplus_{v\vert p} {U}_v$  once tensored with  $\Z_p$  and we have a four-term exact sequence provided by global class field theory    
$$0 \rightarrow D_F  \rightarrow {\bar U_F} \rightarrow  \oplus_{v\vert p} {\bar U}_v \rightarrow \gal({\hat F}/H) \rightarrow 0.$$
First assume  $F$  to be  $p$-rational. Then obviously  $H \subseteq {\tilde F}$ and  $D_F=0$. So applying the snake lemma to the above exact sequence for raising to the $p$-th power maps  shows the other two conditions (ii)  and  (iii).   
Now  assume the three conditions to be fulfilled. 
Consider the commutative diagram where the vertical maps consist of raising to the $p$-th power: 
$$
\begin{array}{cccccccccccc}
0& \rightarrow & D_F  & \rightarrow {\bar U_F} &\rightarrow & \oplus_{v\vert p} {\bar U}_v & \rightarrow & \gal({\hat F}/H) &\rightarrow & 0\\
&&   \downarrow   & \downarrow &&   \downarrow &&   \downarrow \\
0& \rightarrow & D_F  & \rightarrow {\bar U_F} &\rightarrow & \oplus_{v\vert p} {\bar U}_v & \rightarrow & \gal({\hat F}/H) &\rightarrow & 0.
  \end{array}
  $$

A diagram chasing shows that under condition (ii),  $D_F/p$  is contained in the kernel of   
$U_F/p  \to  \oplus_{v\vert p}U_v/p$. This latter map is, in turn, injective by assumption (iii). 
Consequently  $D_F=0$  and, once again, we are faced with a short exact sequence. 
Applying once more the snake lemma to the diagram obtained by the raising to the $p$-th power maps shows that  $\gal({\hat F}/H)$  is torsion free. Hence  $F$  is  $p$-rational. 
% notice that  $\gal({\hat F}/H)$   and   $X_S=\gal({\hat F}/K)$  have the same torsion  T. 
%Then a diagramme chase in the above ex
\end{proof}

%
%We know that the cyclotomic field $\Q(\mu_p)$ is $p$-rational precisely when $p$ is regular, here we prove the same result for its maximal real subfield
%
%\begin{pro}  \label{pregular}  
%
%The real maximal subfield $F=\Q(\mu_p)^+$ of $\Q(\mu_p)$ is $p$-rational precisely when $p$ is regular.
%
%\end{pro}
%
%\begin{proof}  By (\ref{product of p-adic L}), $F$ is $p$-rational when   $$\prod_{i=1}^{(p-3)/2} L_p(1,\omega^{2i})$$   is a $p$-adic unit. Here  $\omega$  is the Teichm\"uller character. On the other hand    $$L_p(1, \omega^{2i}) \equiv L_p(1-2i, \omega^{2i}) \equiv - \frac{B_{2i}}{2i}    \pmod p.$$   where $B_i$ is the $i$th Bernoulli number. Hence $F$ is $p$-rational precisely when $p$ does not divide the numerator of $B_{2i}$ for all $i=1, 2, ..., (p-3)/2,$ which means that $p$ is regular.
%
% \end{proof}
%
%This result shows by the characterisation of $p$-rationality given in Proposition \ref{characterisation} that when $p$ is regular, every unit of $F=\Q(\mu_p)^+$ which is locally a $p$-power, it is also globally. However, this is not the case when $p$ is irregular, since under Vandiver's Conjecture which states that $p$ does not divide the class number of $F$, we get that there exists at least a unit of $F$ which is not globally a $p$-power but locally it is. \\
%

Consider here a totally real number field  $F$  Galois over  $\Q$ and satisfying Leopoldt's conjecture at  $p$. 
If the ramification index  $e$  of the prime  $p$  in  $F/\Q$  is prime to  $p$, then Condition  $(i)$  in the above Proposition \ref{characterisation} is equivalent to  $p\nmid h_F$. Also, if  $p-1\nmid e$,  then Condition  $(ii)$  is fulfilled. 
Therefore, the  $p$-rationality of  $F$  can be read through the class number and the units in the following way: 

\begin{pro} 
\label{ramification index} 
Let  $F$  be a totally real number field which is Galois over  $\Q$ and satisfies Leopoldt's conjecture at the prime  $p$. 
Suppose neither  $p$  nor   $p-1$  divide the ramification index of the prime  $p$  in  $F/\Q$. 
Then  $F$  is  $p$-rational  precisely when the following two conditions are satisfied: \\
(i) The prime  $p$  does not divide the class number  $h_F$ and \\
(ii) The natural map  $U_F/p  \to  \oplus_{v\vert p}U_v/p$  is injective. 
\end{pro} 

Consider now a cyclic extension  $F$  of odd prime degree  $n$  over  $\Q$. 
If  $n=p$,  as explained before, $F$  is  $p$-rational precisely when (apart from  $p$  which may or may not be ramified in  $F$) at most one non-$p$-adic prime  $\ell \not \equiv 1 \pmod {p^2}$  ramifies in  $F$.  For  $n \neq p$, the cyclic field  $F$  satisfies the hypotheses of the above Proposition \ref {ramification index} and we have 

\begin{cor} 
\label{cor cyclic} 
Let  $F$  be a cyclic extension of  $\Q$  of odd prime degree $n \neq p$.  
Then  $F$  is  $p$-rational  precisely when the following two conditions are satisfied: \\
(i) The prime  $p$  does not divide the class number  $h_F$ and \\
(ii) The natural map  $U_F/p  \to  \oplus_{v\vert p}U_v/p$  is injective. 
\end{cor}

When the totally real field  $F$  is quadratic, special attention must be paid to the above condition,  $p-1\nmid e$  when  $p=3$ and is ramified in  $F$,  which we will treat a little further on. Otherwise, the hypotheses of the above Proposition \ref {ramification index} are fulfilled leading to an alternative proof of the following corollary 
 \cite[Proposition 4.1]{Greenberg16}. 

%When the totally real field  $F$  is quadratic and  when either $p\geq 5$ or  $p =3$ 
%and is unramified in $F/\Q$  then the hypotheses of the above Proposition \ref {ramification index} are fulfilled leading to an alternative proof of the following corollary already shown by 
%Greenberg  \cite[Proposition 4.1]{Greenberg16}. 

\begin{cor} 
\label{cor2} 
Let  $F$  be a real quadratic number field. 
Suppose either  $p\geq 5$ or  $p =3$  and is unramified in $F/\Q$. 
Then  $F$  is  $p$-rational  precisely when the following two conditions are satisfied: \\
(i) The prime  $p$  does not divide the class number  $h_F$ and \\
(ii) The fundamental unit of  $F$  is not a p-th power in the completion  $F_v$  of  $F$  at a  $p$-adic prime  $v$. 
\end{cor} 

If instead the quadratic number field  $F$  is imaginary then, the $p$-Hilbert  class field  $H$ being Galois over  $\Q$,  the first condition $H \subseteq {\tilde F}$  in Proposition \ref{characterisation} is equivalent to  $H$ being contained in the anticyclotomic  $\Z_p$-extension of  $F$,  since the layers of the cyclotomic  $\Z_p$-extension of  $F$ are necessarily ramified at the prime(s) above  $p$. Therefore, we also get an alternative proof of the following corollary 
 \cite[Proposition 4.1]{Greenberg16}.

\begin{cor} 
\label{cor3} 
Let  $F$  be an imaginary quadratic number field. 
Suppose either  $p\geq 5$ or  $p =3$  and is unramified in $F/\Q$. 
Then  $F$  is  $p$-rational  precisely when the $p$-Hilbert class field  $H$  is contained in the anticyclotomic  $\Z_p$-extension of  $F$.  
In particular: \\
(i) If  $p$  does not divide the class number  $h_F$, then   $F$  is  $p$-rational. \\
(ii) If  $F$  is  $p$-rational then  $A_F$, the  $p$-primary part of the class group of  $F$,  is at most cyclic.  
\end{cor} 

Many examples exist where the imaginary quadratic number field  $F$  is $p$-rational with a non-trivial  $p$-class group. 
For instance  $\Q(\sqrt{-23})$,  whose class group is of order  $3$,  turns out to be  $3$-rational. Another example is provided by  $\Q(\sqrt{-47})$  which is $5$-rational and whose class group is of order  $5$. 
\vspace{1mm}

%Let us now treat the remaining case, namely when  $p=3$ and  is ramified in the quadratic field  $F=\Q(\sqrt{d})$  where  the integer  $d$  is square-free and a multiple of  $3$. 
Consider now the remaining case when  $p=3$ and  is ramified in the quadratic field  $F=\Q(\sqrt{d})$  (thus the integer  $d$  is square-free and a multiple of  $3$). Then locally, we have   $\Q_3(\sqrt{d})= \Q_3(\sqrt{-3})$,  precisely when  $-d/3$  is a principal unit in the local field  $\Q_3$  which is equivalent to  $d \equiv -3 \pmod 9$. So in this case, again because of  Condition (ii) in Proposition \ref{characterisation}, the field  $F$  is  $3$-rational if and only if  $F=\Q(\sqrt{-3})$.  If instead   $d \not \equiv -3 \pmod 9$, then  Condition (ii) in Proposition \ref{characterisation} is automatically valid and we obtain the same result as before: 
\begin{cor} 
\label{cor4} 
Let  $F=\Q(\sqrt{d})$  be a quadratic field, with  $d \neq -3$ square-free divisible by $3$. \\ 
(i)  If  $d \equiv -3 \pmod 9$  then  $F$  is not  $3$-rational. \\
(ii) If  $d \not \equiv -3 \pmod 9$  and  $d>0$,  then $F$  is  $3$-rational precisely when  $3 \nmid h_F$  and  the fundamental unit of  $F$  is not a third power locally at a  $3$-adic prime  $v$. \\
(iii)  If  $d \not \equiv -3 \pmod 9$  and   $d<0$,  then  $F$  is  $3$-rational   
precisely when the $3$-Hilbert class field  $H$  is contained in the anticyclotomic  $\Z_3$-extension  of  $F$.    
In particular, if  $3$  does not divide the class number  $h_F$, then   $F$  is  $3$-rational.  Also, as before, if $F$  is  $3$-rational then  $A_F$, the  $3$-primary part of the class group of  $F$,  is at most cyclic. 
\end{cor} 

Let  $F=\Q(\sqrt{d})$  be a quadratic field, with  $d\neq -3$ square-free and   $F'=\Q(\sqrt{-3d})$.  Assume that  $d \not \equiv -3 \pmod 9$ which is equivalent to  $3$ not  splitting in $F'$. Suppose first that  $d>0$, we have  
$$ L_3(1, \chi_{{}_F}) \equiv L_3(1-1,\chi_{{}_F})=(1-\chi_{{}_{F'}}(3))L(1-1,\chi_{{}_{F'}})=(1-\chi_{{}_{F'}}(3))\frac{2}{w} h_{F'} \pmod 3, $$  
where $L(\quad ,\chi_{{}_{F'}})$  is the Dirichlet L-function associated to  $\chi_{{}_{F'}}$ and $w=2$ except when $d=3$, in which case $w=4$.  
%On the other hand  $L(1-1,\chi_{{}_{F'}})=-B_{1,\chi_{{}_{F'}}}=h_{F'}$. 
Since $3$ does not split in $F'$,  the factor  $(1-\chi_{{}_{F'}}(3))$ does not vanish and, by Corollary \ref{Lp is unit}, $F$ is $3$-rational precisely when $3\nmid h_{F'}$ (this last equivalence can also be seen using  \cite[Theorem 2]{Browkin85} since, as mentioned before, $F$  is   $3$-rational precisely when  $3$  does not divide the order of the tame kernel $K_2(o_F)$). Now if  $d<0$,  the  $3$-rationality of  $F$  is equivalent to the triviality of  $A'_{F'}$,  the  $3$-primary part of the $3$-class group of  $F'$  \cite[Theorem 4.1]{Fujii08}. Finally, the non splitting of $3$ in  $F'$ also implies that  $A'_{F'} = A_{F'}$  and we obtain the following interesting relationship between the  $3$-rationality of  $F$  and the class number of its mirror quadratic field  (see also \cite[Corollary 4.2]{Greenberg16} and the subsequent paragraph).

%Let  $F=\Q(\sqrt{d})$  be a quadratic field, with  $d$ square-free and   $F'=\Q(\sqrt{-3d})$. 
%Assume that  $d \not \equiv -3 \pmod 9$  so that  $F$  is, according to the above Corollary \ref{cor4}, susceptible to be  $3$-rational. 
%As mentioned before, 
%%The  $3$-rationality of  $F$  can be read through the order of the tame kernel. Namely  
%$F$  is   $3$-rational precisely when  $3$  does not divide the order of  $K_2(o_F)$.  If  $d>0$,  the last condition is equivalent to  $3\nmid h_{F'}$  \cite[Theorem 2]{Browkin85}.  If instead  $d<0$,  the  $3$-rationality of  $F$  is equivalent to the triviality of  $A'_{F'}$,  the  $3$-primary part of the $3$-class group of  $F'$  \cite[Theorem 4.1]{Fujii08}.  
%Now since $d \not \equiv -3 \pmod 9$,  it is not hard to see that the prime  $3$  does not split in  $F'$. 
%This, in particular, implies that  $A'_{F'} = A_{F'}$  and we obtain the following interesting relationship between the  $3$-rationality of  $F$  and the class number of its mirror quadratic field  
%%the prime  $3$  ramifies in  $F'$  unless  $d \equiv 3 \pmod 9$  (recall that $d$  is assumed not to be   $\equiv -3 \pmod 9$). 
%%Summarising, we have  
%(see also \cite[Corollary 4.2]{Greenberg16} and the subsequent paragraph). 

\begin{pro} 
\label{mirror} 
Let  $F=\Q(\sqrt{d})$  be a quadratic field, with  $d \neq -3$ square-free and   $F'=\Q(\sqrt{-3d})$. \\
(i)  If  $d \equiv -3 \pmod 9$  then  $F$  is not  $3$-rational. \\
(ii) If  $d \not \equiv -3 \pmod 9$,  then $F$  is  $3$-rational precisely when  $A_{F'}=0$. 
\end{pro} 

According to  \cite[Theorem A]{Wiles15}, 
%with $S_{-}=\lbrace 3\rbrace$
there exist infinitely many imaginary quadratic fields  $\Q(\sqrt{-d})$  in which  $3$  is inert and with class number prime to $3$.  In particular, these fields  $\Q(\sqrt{-d})$  are  $3$-rational. 
Since  $3d \not \equiv -3 \pmod{9}$, the fields  $\Q(\sqrt{3d})$ (and then also  $\Q(\sqrt{-3}, \sqrt{-d}))$   are  $3$-rational by Proposition \ref{mirror}. 
Hence, we have the following 
\begin{cor}  
\label{infinitely bi-quadratic 3-rational} 
There exist infinitely many imaginary bi-quadratic 3-rational number fields.
\end{cor}
Now a multi-quadratic field is  $p$-rational precisely when all its quadratic subfields are  \cite[Proposition 3.6]{Greenberg16}. 
Therefore combining Corollary \ref{cor3} (i), Corollary \ref{cor4} (iii) and Proposition \ref{mirror}, 
%we see that the degree of a multi-quadratic  $3$-rational field which contains  $\sqrt{-3}$  can not be arbitrarily large. Namely:  
we obtain the following characterisation of multi-quadratic  (hence bi-quadratic) $3$-rational fields containing   $\sqrt{-3}$. 
\begin{cor} 
\label{p=3 multiquadratic} 
Let  $F$  be a real multi-quadratic  $3$-rational field. 
Then  $F(\sqrt{-3})$  is  $3$-rational precisely when  $F$  does not contain any $\sqrt{d}$  with  $d \equiv 1 \pmod 3$.  
In particular a multi-quadratic  $3$-rational field which contains  $\sqrt{-3}$ is at most bi-quadratic. 
%(i) The degree over  $\Q$ of a multi-quadratic  $3$-rational field which contains  $\sqrt{-3}$ is at most  $4$. \\
%(ii) $\Q(\sqrt{-1}, \sqrt{-3})$  is the only multi-quadratic  $3$-rational field which contains both $\sqrt{-1}$ and   $\sqrt{-3}$. 
\end{cor} 

An example of such a bi-quadratic field is simply  $\Q(\sqrt{-1}, \sqrt{-3})$. 

%Let now $F=\Q(\sqrt{-d})$ be an imaginary quadratic field in which $3$ inert and its class number is not divisible by $3$, in particular, $F$ is $3$-rational. If $F'=\Q(\sqrt{3d})$, since $3$ inert in $F$, we have $ 3d \not \equiv -3 \pmod{9}$ and by Proposition \ref{mirror}, $F'$ is also $3$-rational. Hence the bi-quadratic field $F(\sqrt{-3})$ is $3$-rational. Combining this result with the existence of infinitely many such imaginary quadratic field $F$ (see for example \cite[Theorem A]{Wiles15} with $S_{-}=\lbrace 3\rbrace$), we have proved the following:
%
%\begin{cor}  
%There exist infinitely many imaginary bi-quadratic 3-rational number fields.
%\end{cor}

%a comparer avec 
%DIVISIBILITY OF ORDERS OF K2 GROUPS ASSOCIATED TO QUADRATIC FIELDS
%IWAO KIMURA (TOYAMA UNIVERSITY)
%thm 7 et Remark 2.11. OK
Finally, to be thorough, let us briefly treat the case of  $p=2$. 
%Introduce now the subgroup of  $p$-hyperprimary elements of the multiplicative group  $F^ {\cdot}$  of non-zero elements: 
%$$V_S(F):= \{ \alpha \in F^{\cdot} /  (\alpha)={\mathfrak a}^p \quad and \quad \alpha \in K_v^p \quad \forall v\vert p  \}$$
%where   ($\alpha$)  is the principal (fractional) ideal generated by   $\alpha$. 
%
%
%As far as the  $2$-rationality of a quadratic field  $F$  is concerned, the situation is completely different from the odd prime case. Recall first the following arithmetic characterisation of  $p$-rationality when the field in question contains the $p$-th roots of unity  $\mu_p$:  
%
%\begin{pro} 
%\label{characterisation2}  
%%(\cite {Movahhedi-Nguyen90,Movahhedi88,Movahhedi90,AM}) 
%Assume that   $\mu_p \subset F$. Then  $F$  is  $p$-rational precisely when the following two conditions are satisfied: \\
%(i) $F$  contains only one  $p$-adic prime; \\
%(ii) the $p$-primary part of the  $(p)$-class group  $A'_F$  is trivial  (for  $p=2$, one has to replace  $A'_F$   by the narrow class group  $A^{'+}_F$). 
%\end{pro} 
As far as the  $2$-rationality of a quadratic field  $F$  is concerned, the situation is completely different from the odd prime case. Recall first that if the  $p$-th roots of unity  $\mu_p \subset F$, then  $F$  is  $p$-rational precisely when $F$  contains only one  $p$-adic prime and the  $p$-primary part of the  $(p)$-class group  $A'_F$  is trivial  \cite [Th\'eor\`eme et D\'efinition 2.1]{Movahhedi-Nguyen90}. For  $p=2$,   the narrow class group  $A^{'+}_F$  has to replace  $A'_F$.

Therefore  $p$-rationality for  $p=2$  amounts to having one dyadic prime and the vanishing of the  $2$-primary part of the narrow  $(2)$-class group   $A^{'+}_F$.  By the genus formula for the class groups "\`a la Chevalley", it should be possible to draw the list of  $2$-rational quadratic fields. But much more is known. In fact, for  $p=2$,  the above two conditions are also equivalent to the vanishing of the positive \'etale cohomology group  $H^2_+(o_F^\prime,\Z_2(i))$ for any integer  $i\geq 2$ 
(See \cite[Proposition 4.8]{AM} whose proof follows the same ideas as those in the proof of  \cite[Proposition 2.6]{Kolster03}  where the case of real number fields is dealt with).  
This vanishing is characterised for arbitrary finite Galois  $2$-extensions of  $\Q$    
\cite[Proposition 6.4]{AM}: 
\begin{pro} 
\label{positive cohomology} 
Let  $F$  be a finite Galois  $2$-extension of  $\Q$ and  $i \geq 2$. Then the positive
cohomology group  $H^2_+(o_F^\prime, \Z_2(i))$  vanishes (in other words,  $F$  is  $2$-rational) exactly when  $F/\Q$  is unramified outside a set of primes $\{2, \infty, \ell\}$ with  $\ell \equiv \pm 3 \pmod  8$. 
\end{pro} 

So, the quadratic  $2$-rational fields are  $\Q(\sqrt{2})$,  $\Q(\sqrt{-1})$,  $\Q(\sqrt{-2})$  as well as those of the form   
$\Q(\sqrt{\ell})$,  $\Q(\sqrt{2\ell})$, $\Q(\sqrt{-\ell})$ and $\Q(\sqrt{-2\ell})$  where  $\ell$  is any odd prime  $\equiv \pm 3 \pmod  8$. Therefore, obviously the only real multi-quadratic  $2$-rational fields turn out to be  bi-quadratic of the form  $\Q(\sqrt{2}, \sqrt{\ell})$ with 
$\ell \equiv \pm 3 \pmod  8$  and  the only tri-quadratic imaginary $2$-rational fields are of the form  $\Q(\sqrt{-1}, \sqrt{2}, \sqrt{\ell})$ with $\ell \equiv \pm 3 \pmod  8$.  
Note that the same list can also be obtained from \cite[Corollary to Theorem 2]{Gras86}.   

By contrast, we have the following conjecture of Greenberg for odd primes  $p$, which motivated this study: 
\begin{con} (\cite[Conjecture 4.8]{Greenberg16})
\label{Greenberg's conjecture} 
For each odd prime number  $p$,  and each positive integer  $t$, there exists a  $p$-rational number field  
$F$ which is Galois over  $\Q$  with  $\gal(F/\Q) \cong (\Z/2 )^t$. 
\end{con}

As explained by Greenberg \cite[Remark 6.8]{Greenberg16}, if the above conjecture \ref{Greenberg's conjecture} is true for imaginary multi-quadratic fields for a given odd prime  $p$, then there would exist continuous Galois representations 
$\rho : G_{\Q} \to  GL_n(\Z_p)$  
with an open image for all degrees  $n \geq 4$. 

%Though the Cohen-Lenstra heuristic stipulates a positive density for the real quadratic fields whose class numbers are prime to  $p$, to our knowledge, it is not even known if infinitely many such real quadratic fields exist.  Notice that, in view of the corollaries  \ref{cor2} and   \ref{cor4}, if the above conjecture of Greenberg holds, then these real quadratic $p$-rational fields  provide infinitely many real quadratic fields whose class numbers are prime to  $p$.  

In fact, we know when exactly  $p$-rationality is inherited by an arbitrary Galois  $p$-extension  
(\cite {Movahhedi-Nguyen90,Movahhedi88,Movahhedi90}) and this was the strategy to obtain infinitely many non abelian  number fields satisfying Leopoldt's conjecture at a prime  $p$.  For the above conjecture \ref {Greenberg's conjecture}, it would be interesting to investigate going up properties of  $p$-rationality along quadratic (and more generally prime-to-$p$ degrees)  extensions.

\begin{pro} 
\label{going up quadratic field} 
Let  $p>2$ and  $F$  be a totally real  $p$-rational number field and  $L$  a totally imaginary field which is a quadratic extension of  $F$,  a so-called  C.M. field.  Suppose that  $p\nmid h_L$  and  $\mu_p \not \subset L_w$  for any  $p$-adic prime  $w$   of  $L$. Then  $L$  is also  $p$-rational. 
\end{pro} 

\begin{proof}
Since $p$  is odd, the inclusion  $U_F \subset U_L$  induces a natural injective map  $U_F/p \hookrightarrow U_L/p$. Since  $\mu_p \not \subset L$ by hypothesis,  
$U_F/p$  and  $U_L/p$  both have the same  $\Z/p$-rank  $[F:\Q]-1$. 
Therefore, the above natural injective map is in fact an isomorphism. In other words, the Galois group  $\gal(L/F)$ acts trivially on  $U_L/p$.  Consider now the following commutative diagram
$$\xymatrix{
U_F/p \ar[r] \ar[d]^{\simeq} &  \oplus_{v\vert p} U_v/p \ar@{^{(}->}[d] \\
U_L/p \ar[r] &  \oplus_{v\vert p} \oplus_{w\vert v} U_w/p. \\
}$$
Since  $F$  is supposed  $p$-rational, the top horizontal map is injective by Proposition \ref{characterisation}, therefore the bottom horizontal map is also injective and the  $p$-rationality of  $L$ follows by the same Proposition \ref{characterisation}.  
\end{proof}

%The above proposition \ref{going up quadratic field} can be improved in the following way: 
%The prime  $p$ being odd, the Galois group $\gal(\tilde L/L)$  splits into  
%
%Consider the action of the Galois group  $\gal(L/F)$  on  $\gal(\tilde L/L)$. Denote by  $\gal(\tilde L/L)^+$ the subgroup of fixed points and by $\gal(\tilde L/L)^-$ the subgroup on which  $\gal(L/F)$  acts by inversion. 
%The prime  $p$ being odd, we have  $\gal(\tilde L/L)= \gal(\tilde L/L)^+ \times \gal(\tilde L/L)^-$. 

In the above Proposition \ref{going up quadratic field}, the condition  $p\nmid h_L$  can be replaced by the less restrictive condition that the $p$-Hilbert  class field of  $L$  is contained in   ${\tilde L}$. We also note that, for  $p> 3$,  if the C.M. field $L$  in the above proposition \ref{going up quadratic field} is supposed to be multi-quadratic,  then it is  $p$-rational as soon as  $p\nmid h_L$. 
The same result holds for  $p=3$,  provided that  $L$  does not contain any $\sqrt{d}$, where  $d<0$ and  $ \equiv -3 \pmod 9$.  \\

Let $w_2(F)$ be the order of the Galois cohomology group $H^0(F,\Q/\Z(2))$ with coefficients the $2$-fold Tate-twist of the module $\Q/\Z$. The $p$-part of  $w_2(F)$  is the maximal power $p^m$, such that the Galois group $Gal\,(F(\mu_{p^m})/F)$ has exponent $2$. 
If $F$ is totally real and $p$ is odd, the latter condition is equivalent
to the fact that $F$ contains the maximal real subfield
$\Q(\mu_{p^m})^+$ of $\Q(\mu_{p^m})$. 
%Let  $G_F$  be the absolute Galois group of  $F$  and denote by $w_i(F)$  the order of the Galois cohomology 
%group $H^0(G_F,\Q/\Z(i))$ with coefficients the $i$-fold Tate-twist of 
%the module $\Q/\Z$. The $p$-part of  $w_i(F)$  is the maximal power $p^m$, such that
%the Galois group $Gal\,(F(\mu_{p^m})/F)$ has exponent $i$. 
%If $F$ is totally real and $p$ is odd, the latter condition for  $i=2$  is equivalent
%to the fact that $F$ contains the maximal real subfield
%$\Q(\mu_{p^m})^+$ of $\Q(\mu_{p^m})$. 
Then we have the following 
\begin{cor} 
\label{naito} 
Let  $p \geq 7$ and  $F$  be a totally real multi-quadratic  $p$-rational number field. 
Assume that   $p \nmid w_2(F) \zeta_F(-1)$,  where  $\zeta_F$  is the Dedekind $\zeta$-function of  $F$. Then there exist infinitely many totally imaginary quadratic extensions $L_i$  of  $F$  which are  $p$-rational. 
\end{cor} 

\begin{proof}
Since  $F$  is assumed to be multi-quadratic, the odd prime  $p$  ramifies in the layers of the cyclotomic  $\Z_p$-extension of  $F$. Therefore, by Proposition \ref{characterisation} (i),  we have   $p \nmid h_F$. 
Now by  \cite[Theorem] {Naito91}, there exist infinitely many totally imaginary quadratic extensions $L_i$  of  $F$  whose class numbers are primes to  $p$. 
Since  $p$  is supposed $\geq7$, the completion of  $L_i$  at any  $p$-adic prime does not contain the  $p$-th roots of unity and the Corollary follows from Proposition  \ref{going up quadratic field}.   
\end{proof}

The above Corollary  \ref{naito}  also holds for  $p=5$  as soon as  $F$ does not contain any quadratic field  $\Q(\sqrt{d})$,  
where  $d \equiv \pm 5 \pmod {25}$,  since then the completion $F_v$  of  $F$  at any  $p$-adic prime  $v$  would not contain   $\sqrt{5}$  and therefore none of its quadratic extensions could contain  $\mu_5$. 

When  $F$  is totally real and abelian over $\Q$, then 
\begin{equation}\label{birch-tate}
w_2(F) \zeta_F(-1) = \pm {|K_2(o_F)|}
\end{equation}
according to the Birch-Tate conjecture (which is, in this case, a consequence of the Main conjecture in Iwasawa theory proved by Wiles). Hence the above Corollary \ref{naito} in fact applies whenever the totally real multi-quadratic field  $F$  is both $p$-rational and  $p$-regular.  

%Remark that the above Corollary  \ref{naito}  also holds for  $p=5$  provided that $5$  is not a square in  $F$.  As far as  $p=3$ is concerned, we have 

%Let $L/\Q$ be abelian. As a consequence of the dyadic Main Conjecture 
%in Iwasawa-theory
%proved by Wiles [{\bf 18}] for abelian
%fields the Birch-Tate Conjecture is true for $L$ (cf. [{\bf 14}], Appendix A), \ie
%$$\zeta_L(-1) = \pm \frac{|K_2(o_L)|}{w_2(L)},$$
%
%where $w_2(L)$ can be expressed as the order of the Galois cohomology 
%group $H^0(L,\Q/\Z(2))$ with coefficients the $2$-fold Tate-twist of 
%the module $\Q/\Z$. Given an integer $k$, denote by  $\mu_k$ the 
%group of $k$-th roots of unity, then for any prime $p$ the $p$-part 
%of $w_2(L)$ is the maximal power $p^n$, such that
%the Galois group $Gal\,(L(\mu_{p^n})/L)$ has exponent $2$.
%If $L$ is totally real, the latter condition is equivalent
%to the fact that $L$ contains the maximal real subfield
%$\Q(\mu_{p^n})^+$ of $\Q(\mu_{p^n})$ for $p$ odd and the fact that
%$L$ contains the maximal real subfield
%$\Q(\mu_{2^{n-1}})^+$ of $\Q(\mu_{2^{n-1}})$ for $p = 2$.
%Assume now that $E/F$ is a bi-quadratic extension
%of totally real number fields. Since for $m>n$ the extension 
%$\Q(\mu_{p^m})^+/\Q(\mu_{p^n})^+$ is cyclic of degree $p^{m-n}$, we 
%see immediately that $w_2(E)$ and $w_2(F)$ can differ at most by a 
%factor of $2$, and if this is the case, then $w_2(E) = w_2(F_i)$ for 
%precisely one intermediate field $F_i$. In any case we obtain
%$$w_2(F)^2w_2(E) = w_2(F_1)w_2(F_2)w_2(F_3),$$
%and therefore the following Brauer relation for the tame kernels:

\section{The quadratic case} 
%Voir Byeon 2001-2  et  ono99 
There exist infinitely many imaginary quadratic fields whose class numbers are not divisible by a given prime $p$. 
This was a conjecture of Chowla proved by Hartung  \cite {Hartung74}. Since then, many papers have dealt with the non divisibility by  $p$  of the class number of imaginary quadratic fields  $F$  under different conditions. For instance, the same result is proved in  \cite {Horie_Onishi88} or in \cite {Horie90} imposing the behaviour of the prime  $p$ or finite sets of primes in the imaginary quadratic fields in question. In \cite {Kohnen_Ono99}, the authors even show the existence of a lower bound for the number of imaginary quadratic fields with discriminant  $d>-X$   for a large integer  $X$  having class numbers prime to  $p$. 
For more recent treatments see  \cite {Wiles15,Beckwith17}. 

%Many authors studied the non divisibility of the class number of imaginary quadratic fields  $F$  and if we fix the prime  $p$,  there exists infinitely many imaginary quadratic fields whose class numbers are not divisible by $p$ (Chowla's conjecture). 
%This was proved by Hartung 
%See for instance \cite {Hartung74} or the more recent paper \cite {Wiles15}  where, in the imaginary quadratic fields considered, further conditions are imposed on the behaviour of a finite number of rational primes. 

So, according to Corollary \ref{cor3}, for each prime number  $p$  there are infinitely many imaginary quadratic fields which are  $p$-rational. 
%(for instance in \cite {Wiles15},  the prime  $p$ could be chosen unramified in the imagainary quadratic fields so that Corollary \ref{cor3} applies even for  $p=3$). 
Also, we note that the nine imaginary quadratic fields with class number 1 are  $p$-rational for each odd prime  $p$  (in fact they are also  $2$-rational except for  $\Q(\sqrt{-7})$) and it would be interesting to list the (imaginary quadratic) number fields which are  $p$-rational for all odd primes  $p$.  

As far as real quadratic fields  $F$  are concerned, the situation is much harder to handle. 
This is due to the fact that, on one side, in the Dirichlet's class number formula 
\begin{equation}
\label{class number formula}
h_F=   \frac{L(1,\chi_{{}_F}) \sqrt{d_F}}{2R_F},
\end{equation} 
the regulator  $R_F=\log(\varepsilon)$  intervenes so knowledge of the fundamental unit  $\varepsilon$  of  $F$  is required and, on the other hand, for the real quadratic field to be  $p$-rational, the fundamental unit must not be a $p$-th power locally at a prime above  $p$  (Corollaries \ref{cor2} and \ref{cor4}). 

When the odd prime $p \neq 5$,  the  $p$-rationality of  $\Q(\sqrt{5})$  can be connected to  Fibonacci numbers in a particularly pleasant form as explained by Greenberg in  \cite[Corollary 4.5]{Greenberg16}:  let  
$$q: = \left\{
\begin{array}{ll} 
p & \mbox{ if }  p \equiv \pm 1 \pmod 5 ,\\ 
p^2 & \mbox{ if }  p \equiv \pm 2 \pmod 5. 
\end{array}
\right.
$$
Then   $\Q(\sqrt{5})$  is  $p$-rational precisely when  the  $q$-th  Fibonacci number  
$F_q \not \equiv 1 \pmod {p^2}$. 
The Fibonacci numbers are given by the Binet formula 
$$F_q := \frac{\varepsilon_0^q - {\bar \varepsilon_0}^q}{\varepsilon_0 - {\bar \varepsilon_0}},$$
where  $\varepsilon_0:=(1+\sqrt{5})/2$  is the fundamental unit of  $\Q(\sqrt{5})$. 
%The famous Binet formula for the Fibonacci sequence F1 = 1 = F2 , Fn+2 = Fn + Fn+1 is the identity
%?n ? (?1/?)n Fn = �%where ? is the golden ratio (1 + �/2.
Greenberg's proof is extendable similarly to any real quadratic field  $F$  in which the odd prime  $p$  is not ramified. 
Indeed let  $\varepsilon$ be the fundamental unit of  $F$. 
Introduce  $q: =p^f$,  where  $f$  denotes the residue degree of  $p$  in  $F$.  
%$$q: = \left\{
%\begin{array}{ll} 
%p & \mbox{ if }  p \equiv 1 \pmod 4 ,\\ 
%p^2 & \mbox{ if }  p \equiv 3 \pmod 4. 
%\end{array}
%\right.
%$$
Then $\varepsilon^{q-1}$  is a principal unit in the completion  $F_v$  at a  $p$-adic prime  $v$ and is locally a  $p$-th power precisely when it belongs to  $U_v^{(2)}$  \cite [Prop 9, page 219]{Serre62}. 
Therefore,   $\varepsilon$  is locally a  $p$-th power  precisely when  
$\varepsilon^q \equiv \varepsilon \pmod {{\mathfrak p}^2_v}$. 
The same holds for the conjugate  ${\bar \varepsilon}$  which coincides with the inverse $\varepsilon^{-1}$.  
Denote by 
$${\cal F}_n := \frac{\varepsilon^n - {\bar \varepsilon}^n}{\varepsilon - {\bar \varepsilon}}$$ 
the general Fibonacci-type of two-term linear recurrence  
$${\cal F}_{n+2} := (\varepsilon+\bar \varepsilon) {\cal F}_{n+1} - (\varepsilon \bar \varepsilon) {\cal F}_n,$$
with initial terms ${\cal F}_0 =0$  and  ${\cal F}_1 =1 $. 
Accordingly, it is now clear that we have the following corollary which is a consequence of Corollary  \ref{cor2} and the above discussion  inspired by Greenberg's proof of \cite[Corollary 4.5]{Greenberg16}  in the special case of   $\Q(\sqrt{5})$.

\begin{cor} 
\label{Fibonacci} 
Let  $F$  be a real quadratic field in which the odd prime  $p$  does not ramify. 
Then  $F$  is  $p$-rational  precisely when the following two conditions are satisfied: \\
(i) The prime  $p$  does not divide the class number  $h_F$ and \\
(ii) The generalised Fibonacci number  ${\cal F}_q \not \equiv 1 \pmod {p^2}$. 
\end{cor}

This provides an algorithm for testing  $p$-rationality of a real quadratic field (and therefore also a real multi-quadratic field). 
%Using a similar characterisation of  the $p$-rationality of quadratic fields, examples of large real multi-quadratic  $p$-rational fields are obtained in  \cite {Bo} for some prime numbers  $p$. 

Now, if we fix the prime  $p$, we have a real  $p$-rational quadratic field which is explicit: 
\begin{pro} 
\label{realquadratic} 
For each prime  $p\neq3$, the real quadratic field  $F=\Q(\sqrt{p(p+2)})$  is  $p$-rational. For  $p=3$  the corresponding field   $F=\Q(\sqrt{15})$ is not $3$-rational.  
\end{pro} 

\begin{proof}
For  $p=2$, the field in question is  $\Q(\sqrt{2})$, the first layer of the cyclotomic  $\Z_2$-extension of  $\Q$, which is $2$-rational. 
For  $p=3$,  the quadratic field  $\Q(\sqrt{15})$  is not  $3$-rational (Corollary \ref{cor4}). In fact this is because once completed at the  $3$-adic prime it coincides with  $\Q_3(\sqrt{-3})$,  hence condition  (ii)  of Proposition \ref{characterisation} is not fulfilled. 

We may therefore assume  $p>3$.  Denote by  $d$  the square-free part of  $p(p+2)$.  
Since $d\equiv 3 ~ (\mathrm{mod} ~ 4)$, the fundamental unit is of the form  $x+y\sqrt{d}$,  where   $x$  and  $y$  are strictly   positive integers.  We claim that  $\varepsilon:= p+1+\sqrt{p(p+2)}$  is the fundamental unit. Since otherwise there would exist $n\geq2$  such that:  
$$2(p+1)=Tr(\varepsilon)=Tr((x+y\sqrt{d})^n)\geq 2x^n+2x^{n-2}y^2d, $$
where  $Tr:= Tr_{F/\Q}$  is the trace map in  $F/\Q$.  
Since $p\vert d$, we would necessarily have $x=y=1$ and $d=p$. But this is impossible since then  $x+y\sqrt{d}=1+\sqrt{p}$  which is not even a unit. 

Now, the fundamental unit $\varepsilon= p+1+\sqrt{p(p+2)}$  is obviously a principal unit of the completion  $F_v$  of  $F$  at the $p$-adic prime  $v$  and we have 
 $$\varepsilon \in U_v^{(1)}\setminus U_v^{(2)},$$
where  $U_v^{(i)}$  denotes the  $i$-th higher unit group in the local field  $F_v$. 
Therefore  $\varepsilon \not \in  F_v^p$  since raising to the  $p$-th power yields an isomorphism   
\cite [Prop 9, page 219]{Serre62}
$$
\begin{array}{cccc}
 \mu_{p-1}U_v^{(1)}  & \longrightarrow &   \mu_{p-1}U_v^{(3)} \\
 \zeta x & \mapsto & \zeta x^p.\\
\end{array}
$$

It remains to show that  $p$  does not divide the class number  $h_F$. 
%The analytic class number formula for real quadratic fields reads, with usual notations 
%$$h_F=   \frac{L(1,\chi) \sqrt{d_F}}{2R_F}$$
%where $\chi$  is the character associated to our quadratic field  $F$,  $d_F =4d$  is the discriminant and  
%$R_F=\log(p+1+\sqrt{p(p+2)})$  is the regulator. 
Since  $2$ ramifies in $F$, we also have  \cite [Corollary 2]{Louboutin04} 
 $$ L(1,\chi_{{}_F}) \leq \frac{\log{d_F}+\kappa_2}{4},  $$ 
where $\kappa_2:=2+\gamma-\mathrm{log}(\pi)=1.432...$, and $\gamma=0.577...$ denotes Euler's constant. 

For our quadratic field  $F$,  the discriminant  $d_F =4d$  and  $R_F=\log(p+1+\sqrt{p(p+2)})$  is the regulator.  
Consequently, by class number formula  (\ref{class number formula}), we successively have: 

%$$ \begin{tabular}{llll}
%$h_F $& $< $ & $\frac{\log(4p(p+2))+2}{4} \frac{\sqrt{4p(p+2)}}{2\log(p+1+\sqrt{p(p+2)})} $\\ 
%
%& $< $ & $ \frac{\log(4p(p+2))+2}{4} \frac{\sqrt{p(p+2)}}{\log(2\sqrt{p(p+2)})} $\\ 
%
%& $< $ & $ \frac{\log(4p(p+2))+2}{2\log(4p(p+2))} \sqrt{p(p+2)}$\\ 
%
%& $< $ & $ (\frac{1}{2}+ \frac{1}{3}) \sqrt{p(p+2)}$\\ 
%
%& $< $ & $ p $
%\end{tabular} $$

$$ \begin{array}{llll}
 h_F  &  <   &  \frac{\log(4p(p+2))+2}{4} \frac{\sqrt{4p(p+2)}}{2\log(p+1+\sqrt{p(p+2)})}  \\ 
 \\
&  <   &   \frac{\log(4p(p+2))+2}{4} \frac{\sqrt{p(p+2)}}{\log(2\sqrt{p(p+2)})}  \\ 
 \\
&  <   &   \frac{\log(4p(p+2))+2}{2\log(4p(p+2))} \sqrt{p(p+2)} \\ 
 \\
&  <   &   (\frac{1}{2}+ \frac{1}{3}) \sqrt{p(p+2)} \\ 
 \\
&  <   &  p 
\end{array} $$
The above inequality  $ h_F<p$  can also be obtained using  \cite[Corollary 3]{Ramare01}. 
Hence  the prime  $p$  does not divide  $h_F$  and the proof is complete.  
\end{proof}

Let  $p>3$  be an odd prime. If  $p \neq 2{a^2} \pm 1$  for any rational integer  $a$,  then the quadratic field  
$\Q(\sqrt{p^2-1})$  is  $p$-rational  \cite[Lemma 2.2]{BR}. The condition  $p \neq 2{a^2} \pm 1$  is needed to ensure the fact that  $p+\sqrt{p^2-1}$  is  the fundamental unit of  $\Q(\sqrt{p^2-1})$,  whereas in the situation of the above proposition, 
$p+1+\sqrt{p(p+2)}$   turns out to always be the fundamental unit.  
Nevertheless, our method also applies to prove the  $p$-rationality of  $\Q(\sqrt{p^2-1})$ whatever the prime  $p$. 

Now fix a prime number  $p>3$  and  let  $d \neq (-1)^{\frac{p-1}{2}}p$  be a square-free integer such that  
$(-1)^{\frac{p-1}{2}}d >0$.   
Introduce the real quadratic field  $F:=\Q(\sqrt{(-1)^{\frac{p-1}{2}}pd})$  as well as the quadratic field  $K:=\Q(\sqrt{d})$.  
Then  $\chi_{{}_K}$  and  $\frac{p-1}{2}$  are of the same parity so that  $L(1-\frac{p-1}{2},\chi_{{}_K})$  is a non-zero integer. 
We also have \cite[Chapter 5]{Washington97}: 
\begin{equation}\label{congruences}
\begin{array}{llll}
L_p(1,\chi_{{}_F}) &  \equiv & L_p(1-\frac{p-1}{2},\chi_{{}_F}) &  \pmod p \\ 
 \\
& \equiv & L(1-\frac{p-1}{2}, \chi_{{}_F} \omega^{\frac{p-1}{2}})  & \pmod p \\ 
 \\
& \equiv & L(1-\frac{p-1}{2},\chi_{{}_K}) &  \pmod p 
\end{array} 
\end{equation} 

where  $\omega$  is the Teichm\"uller character. 
Thus, according to Corollary \ref{Lp is unit}, the quadratic field  $F$  is  $p$-rational precisely when   $L(1-\frac{p-1}{2},\chi_{{}_K})$  is  a  $p$-adic unit.  
Now take for  $d$  the square-free part of  $(-1)^{\frac{p-1}{2}}(p+2)$. 
Notice that if  $p+2$  is a square then  $p \equiv 3 \pmod 4$  so that  $d$  is never a square, whatever the prime  $p$. 
Then, for this choice of  $d$, the field  $F$  is  $p$-rational  (Proposition  \ref{realquadratic}) as soon as  $p>3$  and accordingly  
$$L(1-\frac{p-1}{2}, \chi_{{}_K})$$    
(so also the generalised Bernoulli number   $B_{\frac{p-1}{2},\chi_{{}_K}}$)  
is a  $p$-adic unit. 
Therefore the hypotheses of the following theorem are fulfilled with  $D_0$  being the discriminant of $K=\Q(\sqrt{d})$:

\begin{theorem} \cite[Theorem 2]{Ono99}. 
\label{Ono} 
Let  $p>3$  be prime, and suppose there is a fundamental discriminant  $D_0$  co-prime to  $p$ for which 
\begin{enumerate}[(i)]
\item  $(-1)^{\frac{p-1}{2}}D_0>0$
\item  $\vert B_{\frac{p-1}{2}, \chi_{{}_{D_0}}} \vert _p =1$. 
\end{enumerate}
Then there is an arithmetic progression  $r_p  \pmod {t_p}$  with  $(r_p,t_r)=1$  and  a constant  $\kappa(p)$  such that for each prime  $\ell  \equiv r_p  \pmod {t_p}$  there is an integer  $1\leq d_\ell \leq \kappa(p) \ell$,  for which 
\begin{enumerate}[(i)]
\item  $D_\ell := d_\ell \ell p$  is a fundamental discriminant, 
\item  $h_{\Q(\sqrt{D_\ell})} \not\equiv 0  \pmod p$, 
\item $\vert \frac{R_p(D_\ell)}{\sqrt{D_\ell}}\vert _p =1$.  
\end{enumerate}
%$$\# \{0<D<X ;  p \nmid h_{\Q(\sqrt{D})}, \; p\vert D \;  and \;  \vert \frac{R_p(D)}{\sqrt{D}}\vert _p =1 \} \gg_p
%\frac{\sqrt{X}}{\log X} $$ 
\end{theorem} 

In the above theorem of Ono,  
%$D_\ell$  is a fundamental discriminant,  
$R_p(D_\ell):=R_{p, \Q(\sqrt{D_\ell})}$  is the  $p$-adic regulator of  $\Q(\sqrt{D_\ell})$  and  $\vert \quad \vert_p$  is the  $p$-adic absolute value normalised by  $\vert p \vert_p= 1/p$. 
%Finally  $\gg_p$  means 

Now according to Formula  (\ref{sgn}),  the real quadratic number fields  $\Q(\sqrt{D_\ell})$  with the fundamental discriminants  $D_\ell$  intervening in the above theorem turn out to be all  $p$-rational. Summarising, we proved the following: 

\begin{theorem} 
\label{infinitely many real quadratic} 
%(to be compared with \cite[Theorem 1.1]{Byeon03})
For each prime  $p>3$, there exist infinitely many (with positive density) real quadratic $p$-rational fields in which  $p$ ramifies. 
\end{theorem} 

What we have just explained is inspired by a paper of Byeon \cite[Theorem 1.1]{Byeon03}. 
He does not consider the same quadratic field  $F$ and his proof is longer and more complicated. 
%, but we believe that the strategy is essentially the same. 
For  $p=3$,  there also exist infinitely many real quadratic  $3$-rational fields in which $3$ ramifies (Corollary \ref {infinitely bi-quadratic 3-rational}).

\begin{remark} 
%%\label{markus} 
%1. We also know \cite[Corollary 1.4]{Sc16}  that for all positive integer $m$ and for almost all primes  $p \equiv 3 \pmod 4$ with $(p,m)=1$, there exist infinitely many negative fundamental discriminants  $-d$ prime to $m$ such that  $L(2-\frac{p+1}{2}, \chi_{{}_{-d}})$  is a  $p$-adic unit. 
%This also provides a way to prove the following : fix a finite set  $S$  of rational odd prime numbers. Then, by the congruences (\ref{congruences}), for almost all primes  $p \equiv 3 \pmod 4$ there exist  infinitely many real quadratic $p$-rational fields in which all the primes in  $S$  are unramified. \\
%We also know \cite[Corollary 1.4]{Sc16}  that for all but finitely many primes  $p \equiv 3 \pmod 4$, there exist infinitely many negative fundamental discriminants  $-d$  such that  $L(2-\frac{p+1}{2}, \chi_{{}_{-d}})$  is a  $p$-adic unit. For each such  $d$, introduce the real quadratic field  $F:=\Q(\sqrt{dp)})$.  
%On the other hand, we have \cite[Theorem 5.24]{Washington97} 
%$$ L_p(1,\chi_{{}_F})  = (1- \frac{\chi_{{}_F}(p)}{p}) 2 h_F R_{p,F} / \sqrt{d_F}$$
%Since  
%$$L_p(1,\chi_{{}_F}) \equiv L_p(1-\frac{p-1}{2},\chi_{{}_F})  \equiv L(1-\frac{p-1}{2},\omega^{-\frac{p-1}{2}} \chi_{{}_F})  
%\equiv L(2-\frac{p+1}{2}, \chi_{{}_{-d}})  \pmod p$$
%we obtain the  $p$-rationality of  $F$  by  Coates Formula (\ref{coates}). Finally, we get infinitely many  $p$-rational real quadratic fields for all but finitely many primes  $p \equiv 3 \pmod 4$. 
%2. 
A proof of the existence of infinitely many real quadratic  $5$-rational fields using half-integral weight modular forms can be found in  \cite{AB}. 
\end{remark}

%On the other hand, for any fixed real quadratic field  $F$ and any integer  $X$,  it is proved \cite[Corollary]{maire-Rougnant}, under the generalised  $abc$-conjecture, that the number of primes  $p\leq X$  such that  $F$  is  $p$-rational is at least  $clog(X)$  for a constant  $c$  depending on  $F$. 
%("Greatly inspired by the computations of Silverman")

We finish this section by making a few observations about the special case of  $p=5$. 
  
\begin{pro} 
\label{p=5} 
Let  $d \neq 5$  be a positive square-free integer.  
Then  $\Q(\sqrt{5d})$  is  $5$-rational precisely when   $\Q(\sqrt{d})$ is  $5$-regular. 
\end{pro}

\begin{proof}
Let  $F=\Q(\sqrt{5d})$  and  $K=\Q(\sqrt{d})$.  
By definition,  $K$  is  $5$-regular precisely when  $5$  does not divide the order of the tame kernel  $K_2(o_K)$. 
Since  $2$ and  $3$  are the only prime divisors of  $w_2(K)$  and  $\zeta_{\Q}(-1)=\frac{-1}{12}$, 
the last condition means, by the equality (\ref{birch-tate}), that  $\zeta_K(-1)$  or equivalently   $L(-1,\chi_{{}_K})$ is a  $5$-unit. 
Now, by the congruences (\ref{congruences}), the last condition is, in turn, equivalent to  $L_5(1,\chi_{{}_F})$  being a  $5$-unit. 
Finally Corollary \ref{Lp is unit} completes the proof. 
\end{proof}

\begin{Ex} 
The quadratic field  $\Q(\sqrt{35})$  is both $5$-rational (Proposition \ref{realquadratic})  and  $5$-regular (Proposition \ref{p=5}). 
Therefore, by the discussion following the proof of  Corollary \ref{naito}, there exist infinitely many imaginary  $5$-rational number fields of degree $4$  over  $\Q$  containing  $\Q(\sqrt{35})$.
\end{Ex}

Let $F/k$ be a bi-quadratic extension of totally real number fields with quadratic subfields $k_1,
k_2,k_3$, such that $F/\Q$ is abelian. Then we have the following Brauer relation for the tame kernels \cite[Proposition 1.1]{KM03} 
$$|K_2(o_k)|^2 |K_2(o_F)| = |K_2(o_{k_1})||K_2(o_{k_2})||K_2(o_{k_3})|.$$
Now, let   $F$  be a real multi-quadratic field of degree  $n=2^m$. 
Since the tame kernel  $K_2(\Z)$ of  $\Q$  is of order  $2$  \cite[Corollary 10.2]{Milnor71}, an induction on the integer $m$ readily yields the following formula  
$$2^{n-2} |K_2(o_F)| = \prod |K_2(o_k)|,$$
%up to a power of  $2$/ 
where the product is taken over all quadratic subfields of  $F$. 
In particular, the real multi-quadratic field  $F$  is  $p$-regular for an odd prime  $p$   precisely when it is the case of all its quadratic subfields. 
Recall that the same holds for the  $p$-rationality of  $F$. 
Therefore the above discussion, together with Proposition \ref{p=5} and the fact that  $\Q(\sqrt{5})$  is  $5$-rational, yields the following: 
 
\begin{cor} 
\label{p=5 multiquadratic} 
Let  $F$  be a real multi-quadratic field which is both  $5$-rational and  $5$-regular. Then  $F(\sqrt{5})$  is  $5$-rational (and  $5$-regular). 

\end{cor}
%In the next section, we will see an alternative proof of the above corollary, valid for an arbitrary prime  $p \equiv 1 \pmod 4$. 

\section{The bi-quadratic case} 
In this section, we will prove that for each prime  $p$  there exist imaginary and real bi-quadratic  $p$-rational fields. Recall that a multi-quadratic field is  $p$-rational precisely when all its quadratic subfields are  \cite[Proposition 3.6]{Greenberg16}. 

\subsection {Imaginary bi-quadratic case} 
The class number of an imaginary quadratic field can be expressed by means of the Kronecker symbol  
$ \left(\frac{\; \cdot \;}{\; \cdot \;}\right) $:   

\begin{theorem} 
\cite[Theorem 219]{Landau58} 
%page 
\label{Landau}
Let $d_F$ be the fundamental discriminant of an imaginary quadratic field $F$. Then 
%$$ h_F=\frac{\vert\mu(F)\vert}{2[2-\left(\frac{d_F}{2}\right)]}\sum_{m=1}^{\lceil\vert d_F \vert /2\rceil} \left(\frac{d_F}{m}\right), $$
$$ h_F=\frac{w}{2[2-\left(\frac{d_F}{2}\right)]}\sum_{r} \left(\frac{d_F}{r}\right), $$
where  $w =2,4$ or $6$  stands for the number of roots of unity in $F$. 
The sum runs over all integers  $r$  from  $1$  to  $-d_F/2$. 
 \end{theorem}

As an immediate consequence of the above class number formula, we see that  $p$  does not divide the class number of the imaginary quadratic fields  $\Q(\sqrt{-p})$  and  $\Q(\sqrt{-p-2})$, which are therefore  $p$-rational for any  $p>3$ (Corollary \ref{cor3}). Now, as an immediate consequence of  \cite [Proposition 3.6]{Greenberg16}  and Proposition  \ref{realquadratic}, we obtain 
\begin{pro} 
\label{imaginary bi-quadratic} 
For any prime  $p>3$, the imaginary bi-quadratic field  $\Q(\sqrt{p(p+2)}, \sqrt{-p})$  is  $p$-rational.  

\end{pro} 
On the other hand,  $\Q(\sqrt{2}, \sqrt{-1})=\Q(\mu_8)$  is  $2$  and  $3$-rational.  
We also note that when  $p+2$  is a square, then the field in question is just  $\Q(\sqrt{p}, \sqrt{-1})$, which turns out to be  $p$-rational. 
The same holds when  $p-2$  is a square (Proposition \ref{real bi-quadratic}) and also when $p-1$ or $p-4$  is a square since our methods also lead to the  $p$-rationality of  $\Q(\sqrt{p(p-1)})$  and  $\Q(\sqrt{p(p-4)})$ in a similar manner.  
%In a similar manner, $\Q(\sqrt{p(p-2)}, \sqrt{-p})$  also turns out to be  $p$-rational for  $p>3$. 
One may then ask the following: 

\begin{question} 
\label{Question}
%Does there exist a prime  $p$  for which   $\Q(\sqrt{p})$  is not  $p$-rational? 
Is  $\Q(\sqrt{p})$  always  $p$-rational?
\end{question} 

%We do not know of a single such prime  $p$.  
We do not know of a single prime  $p$  for which  $\Q(\sqrt{p})$  is not  $p$-rational.  
By the same upper bound as in Proposition  \ref{real bi-quadratic}, $p$ does not divide the class number of  $\Q(\sqrt{p})$  and in some sense the above question concerns its fundamental unit.  
We also note that, since the cyclotomic field  $\Q(\mu_p)$  is  $p$-rational for a regular prime, it is also the case of all its subfields. 
Therefore, if such a prime  $p$  exists, it is either irregular or  $\equiv 3 \pmod 4$.

\medskip

\subsection {Real bi-quadratic case} 

\begin{pro} 
\label{real bi-quadratic} 
For each prime  $p>3$, the real bi-quadratic field  $\Q(\sqrt{p(p+2)}, \sqrt{p(p-2)})$  is  $p$-rational.  
\end{pro} 
\begin{proof}
The difference  $p^2-4$ can never be a square so we are really dealing with a bi-quadratic field. 
The proof for the  $p$-rationality of   $\Q(\sqrt{p(p-2)})$ goes along the same line as that of Proposition  \ref{realquadratic} except that the non divisibility of its class number by  $p$  is easier to establish since  $p(p-2)<p^2$ (see for instance \cite[Theorem (a)]{Le94}).  
Note that the quadratic field  $\Q(\sqrt{p(p-2)})$  has been considered in  \cite{Byeon03}. 
%and  a detailed analysis of the study there together with the above formula (\ref{sgn}) leads to the  $p$-rationality of  $\Q(\sqrt{p(p-2)})$. 
Let us now prove the  $p$-rationality of the third quadratic subfield  $k:=\Q(\sqrt{p^2-4})$. 
Introduce   
$$q: = \left\{
\begin{array}{ll} 
p & \mbox{ if }  p \equiv 1 \pmod 4 ,\\ 
p^2 & \mbox{ if }  p \equiv 3 \pmod 4 ,
\end{array}
\right.
$$
and the unit  $\varepsilon :=\frac{1}{2}(p+\sqrt{p^2-4})$. Then,  
$$  \begin{array}{llll} 
       (2\varepsilon)^{q-1} &=(p+\sqrt{p^2-4})^{{q-1}} &\\
       &\equiv  (p^2-4)^{\frac{q-1}{2}}+p(q-1)(p^2-4)^{\frac{q-3}{2}}\sqrt{p^2-4}  &  \pmod {p^2 o_k}\\
        & \equiv (-4)^{\frac{q-1}{2}}-(-4)^{\frac{q-3}{2}} p \sqrt{p^2-4}  &  \pmod {p^2 o_k}\\
        & \equiv 2^{q-1}+2^{q-3}p\sqrt{p^2-4}  &  \pmod {p^2 o_k}
\end{array}$$
where  $o_k$  denotes the ring of integers of  $k$. Accordingly  $\varepsilon \in U_v^{(1)}\setminus U_v^{(2)}$ 
where again  $U_v^{(i)}$  denotes the  $i$-th higher unit group in the local field  $k_v$ for a  $p$-adic prime  $v$  of  $k$. 
This guarantees the fact that our unit  $\varepsilon$  (hence, also the principal unit) is not a  $p$-th power locally at  $v$. Let us now look at the class number $h_k$. It is possible to prove  $h_k<p$  using again  \cite [Corollary 2]{Louboutin04} 
%Since  $2$  is inert in  $k$, we have the following inequality  \cite [Corollary 2]{Louboutin04}
%$$ L(1,\chi) \leq \frac{\log{(d_k)}+\kappa_3}{6}, $$ 
%where the discriminant  $d_k$  is the square-free part of  $p^2-4$,  $\kappa_3:=2+\gamma-\mathrm{log}(\pi/4)=2.818...$ and $\gamma=0.577...$ denotes Euler's constant.  On the other hand, the fundamental unit $\varepsilon_0$  of  $k$ satisfies the inequality   $\varepsilon_0 \geq 1+ \sqrt{d_k}$  so that, by the analytic class number formula, we have 
%$$h_k = \frac{L(1,\chi) \sqrt{d_k}}{2R_k} < \frac{\log{(d_k)}+\kappa_3}{6} \frac{\sqrt{d_k}}{2\log(\sqrt{d_k})} < \frac{\log{(d_k)}+3}{6\log{(d_k)}}  \sqrt{d_k} < \sqrt{d_k} < p$$
but here we have a better upper bound. Namely, 
by  \cite[Theorem (a)]{Le94} or  \cite[Corollary 2]{Ramare01}, we have the following inequalities 
$$h_k \leq \frac{1}{2} \sqrt{d_k} \leq \frac{1}{2} p.$$
%
%$$ \begin{array}{llll}
%h_k & = & \frac{L(1,\chi) \sqrt{d_k}}{2R_k} \\ 
%& < & \frac{\log{(d_k)}+\kappa_3}{6} \frac{\sqrt{d_k}}{2\log(\sqrt{d_k})} \\
%& < & \frac{\log{(d_k)}+3}{6\log{(d_k)}}  \sqrt{d_k} \\
%& < & \sqrt{d_k} \\
%& <  &  p 
%\end{array} $$
%
Hence  the prime  $p$  does not divide  $h_F$  and the field  $k$  is  $p$-rational by Corollary \ref{cor2}. The proof is complete.  
\end{proof}
%It should be noticed that the quadratic field 

%By the same method, one could show the  $p$-rationality of similar bi-quadratic fields such as 
%$$\Q(\sqrt{p(p+4)}, \sqrt{p(p-4)}).$$ 

%$\Q(\sqrt{p(p+r)}, \sqrt{p(p-r)})$   with $r=1$ or  $4$. 

For  $p=3$, the above corresponding bi-quadratic field  $\Q(\sqrt{3}, \sqrt{5})$  is not  $3$-rational since as noticed before   
$\Q(\sqrt{15})$  is not  $3$-rational.  
The field  $\Q(\sqrt{2}, \sqrt{5})$  is both  $2$ and $3$-rational 
(Proposition \ref{positive cohomology} and  and  Corollary \ref{cor2}). 

Note that, by the same method, we can prove the  $p$-rationality of the real quadratic fields

\begin{Ex} 
By the above Proposition \ref{real bi-quadratic}, the bi-quadratic field  $\Q(\sqrt{15}, \sqrt{35})$  is $5$-rational. 
On the other hand, the three quadratic subfields  $\Q(\sqrt{15})$, $\Q(\sqrt{21})$  and  $\Q(\sqrt{35})$  are $5$-regular by  Proposition \ref{p=5},  hence  $\Q(\sqrt{15}, \sqrt{35})$  is also $5$-regular  \cite[Proposition 1.1]{KM03}. 
%Therefore there exist infinitely many imaginary  $5$-rational number fields of degree $8$  over  $\Q$  containing  $\Q(\sqrt{15}, \sqrt{35})$. 
Therefore there exist infinitely many  $5$-rational CM fields having the same maximal real subfield  $\Q(\sqrt{15}, \sqrt{35})$.
\end{Ex}

Let  $\alpha$  be a prime-to-$p$  positive integer and consider the real bi-quadratic field 
$$K_\alpha := \Q( \sqrt{\alpha p(p + 2)},  \sqrt{\alpha p(p-2)}).$$ 
The same arguments as those in the proof of Proposition \ref{realquadratic} and the above Proposition \ref{real bi-quadratic} show  that the fundamental unit of each subfield of  $K_\alpha$  is not locally a  $p$th-power at the $p$-adic primes. So  $K_\alpha$  is  $p$-rational as soon as  $p\nmid h_{K_\alpha}$.  Though there exist  $\alpha$  such that  $p\mid h_{K_\alpha}$  (for instance with  $p=5$  and  $\alpha =17$),  it would be interesting to find an infinite family of integers   $\alpha$  for which  $p\nmid h_{K_\alpha}$.

\subsection {A step further} 
Having the real bi-quadratic  $p$-rational field   $\Q(\sqrt{p(p+2)}, \sqrt{p(p-2)})$  at hand, the first question which springs to mind is how to add an imaginary  square root in order to obtain a tri-quadratic $p$-rational field. We already noticed that the imaginary quadratic fields  
$\Q(\sqrt{-p})$  and  $\Q(\sqrt{-p-2})$  are   $p$-rational. 
Likewise,   $\Q(\sqrt{-p+2})$  is  $p$-rational since  $p$  does not divide its class number according to the above Theorem  \ref{Landau}. Therefore, by Proposition \ref{real bi-quadratic}, the tri-quadratic number field   
$$L:=\Q(\sqrt{p(p+2)}, \sqrt{p(p-2)}, \sqrt{-p}) = \Q(\sqrt{-p-2}, \sqrt{-p}, \sqrt{-p+2})$$  
is  $p$-rational precisely when the imaginary quadratic field $\Q(\sqrt{-p(p^2-4)})$  is   $p$-rational.  
Using PARI/GP, the smallest prime number dividing its class number is   $p_1:=192.699.943$. 
Thus, at least for every prime  $p<p_1$, the above field 
%one sees that  $p_1=192.699.943$  is the only prime number $<192.699.943$ which divide its class number. 
%Gras's algorithm for testing  $p$-rationality shows that the above imaginary quadratic field remains  $p$-rational even for  $p=p_1$. 
%Thus, at least, for every prime  $p<722.500.879$    
$L$  is $p$-rational. It would be interesting to investigate its  $p$-rationality for an arbitrary prime  $p$.  
It is also worth mentioning that the existence of the above tri-quadratic  $p$-rational field together with the method developed by Greenberg  \cite[Proposition 6.1 and Remark 6.8]{Greenberg16}  guarantee the existence of continuous Galois representations  
$$\rho : G_{\Q} \to  GL_n(\Z_p)$$  
of the absolute Galois group  $G_{\Q}$  of  $\Q$  for both  $n =4$  and  $n =5$  and for all  $p <192.699.943$  with an open image.   

Another natural choice is to add the square root of  $-1$  to the real bi-quadratic  $p$-rational field  $\Q(\sqrt{p(p+2)}, \sqrt{p(p-2)})$. Then the  $p$-rationality of the tri-quadratic field 
$$\Q(\sqrt{p(p+2)}, \sqrt{p(p-2)}, \sqrt{-1}) $$  
is equivalent to that of the three imaginary quadratic subfields  other than  $\Q(\sqrt{-1})$
%$$\Q(\sqrt{4-p^2}), \Q(\sqrt{2p-p^2})  and  $\Q(\sqrt{-2p-p^2})$$ 
which is also interesting to investigate.

In any case, we believe that for each odd prime  $p>3$,  there exists a square-free positive integer $d_p$ such that the relative class number of  
$$\Q(\sqrt{p(p+2)}, \sqrt{p(p-2)},\sqrt{-d_p})$$  
is not divisible by $p$. 
Such a tri-quadratic field is then  $p$-rational by Propositions \ref{real bi-quadratic} and \ref{going up quadratic field}. 
Therefore, we would have the above Galois representations  
$$\rho : G_{\Q} \to  GL_n(\Z_p),$$  
for both  $n =4$  and  $n =5$.

\begin{acknowledgements}\label{ackref}
We would like to express our gratitude to Ralph Greenberg for his interest and suggestions which were encouraging and led to the improvement of an early version of this paper. 
\end{acknowledgements}

\affiliationone{% in this example, two authors share an institution
   XLIM UMR 7252 CNRS \\
  Math\'ematiques \\
  Universit\'e de Limoges\\
  123 Avenue Albert Thomas \\
  87060 LIMOGES Cedex\\
  France
   \email{youssef.benmerieme@unilim.fr\\
   mova@unilim.fr}}
% Important: Do not put any empty line here.

\end{document}